\documentclass[a4paper,11pt]{amsart}

\usepackage{enumerate}

\usepackage{upref}

\usepackage[colorlinks,backref]{hyperref}                   

\usepackage{amssymb,amsmath}

\newtheorem{theorem}{Theorem}[section]
\newtheorem{lemma}[theorem]{Lemma}

\newtheorem{definition}{Definition}[section]

\theoremstyle{remark} 
\theoremstyle{definition} 

\numberwithin{equation}{section}

\newif\ifcomment \commentfalse \def\commentON{\commenttrue}
 \long\outer\def\BC#1\EC{\ifcomment \sloppy \par
\# \ldots\dotfill {\em #1} \dotfill \# \par \fi } \commentON

\newcommand{\remove}[1]{}

\newcommand{\R}{{{\mathbb{R}}}}

{%

\begin{enumerate}}%
{\end{enumerate}
}

\begin{document}

\title[SHARP BLOW-UP CONDITION FOR A  CROSS ATTRACTION SYSTEM]{ SHARP BLOW-UP CONDITION FOR A TWO-SPECIES CROSS ATTRACTION SYSTEM IN 3 OR MORE DIMENSIONS}

\author[{S. Ulusoy}]{Suleyman Ulusoy}
\address[ Suleyman Ulusoy]{\newline
	Department of Mathematics and Natural Sciences\newline
        American University of Ras Al Khaimah\newline
	PO Box 10021, Ras Al Khaimah, UAE}
\email{suleyman.ulusoy@aurak.ac.ae}
\urladdr{https://aurak.ac.ae/en/dr-suleyman-ulusoy/\newline
$https://www.researchgate.net/profile/Suleyman_{-}Ulusoy$}




\date{\today}

\begin{abstract}
We consider a degenerate chemotaxis model with two-species and
two-stimuli in dimension $d \geq 3$. Under the hypothesis of integrable initial data with finite second moment and energy, we show local-in-time existence for any mass of free-energy solutions, namely weak solutions with some free energy estimates. We exhibit that the qualitative behavior of solutions is decided by a set of  critical values: there is a critical value of a parameter pair  in the system of  equations for which there is a global-in-time energy solution and   there exist blowing-up free-energy solutions under a criticality condition is violated for the  parameter pair.

\end{abstract}

\maketitle


\section{Introduction}\label{sec:intro}

Inspired by \cite{carrillo2022} (see also \cite{espejo2019, knutsdottir2014}), for modeling the interaction and motion of two cell populations in breast cancer cell invasion models in $\mathbb{R}^d$, $d \ge 3$, we propose the following chemotaxis kind of system with two chemicals and a nonlinear diffusion term:
\begin{equation}\label{1.1}
\begin{cases}
\partial_t u = \alpha_1 \left( \frac{d - 2}{d} \right) \|u\|_{\frac{2d}{d+2}}^{\frac{4}{d+2}} \Delta\left(u^{\frac{2d}{d+2}}\right) - \nabla \cdot (u\nabla v), & x \in \mathbb{R}^d, \; t > 0, \\[2mm]
-\Delta v = w, & x \in \mathbb{R}^d, \; t > 0, \\[2mm]
\partial_t w = \alpha_2 \left( \frac{d - 2}{d} \right) \|w\|_{\frac{2d}{d+2}}^{\frac{4}{d+2}} \Delta\left(w^{\frac{2d}{d+2}}\right) - \nabla \cdot (w\nabla z), & x \in \mathbb{R}^d, \; t > 0, \\[2mm]
-\Delta z = u, & x \in \mathbb{R}^d, \; t > 0,
\end{cases}
\end{equation}
with initial data satisfying
\begin{equation}\label{1.2}
u(x, 0) = u_0(x); \quad w(x, 0) = w_0(x), \quad x \in \mathbb{R}^d.
\end{equation}
In \eqref{1.1}, $u(x, t)$ and $w(x, t)$ denote the macrophages and the tumor cells; $v(x, t)$ and $z(x, t)$ denote the concentration of the chemicals produced by $w(x, t)$ and $u(x, t)$, respectively, and $\alpha_1, \alpha_2 \in (0, 1]$. For simplicity, the initial data are assumed to satisfy
\begin{equation}\label{1.3}
u_0 \in L^1\left(\mathbb{R}^d; (1 + |x|^2)dx\right) \cap L^\infty(\mathbb{R}^d), \quad \nabla u_0^{\frac{2d}{d+2}} \in L^2(\mathbb{R}^d), \quad u_0 \ge 0,
\end{equation}
\begin{equation}\label{1.4}
w_0 \in L^1\left(\mathbb{R}^d; (1 + |x|^2)dx\right) \cap L^\infty(\mathbb{R}^d), \quad \nabla w_0^{\frac{2d}{d+2}} \in L^2(\mathbb{R}^d), \quad w_0 \ge 0.
\end{equation}

We note that $v(x, t)$ and $z(x, t)$ are given by
\begin{equation}\label{1.5}
v(x, t) = K * w = c_d \int_{\mathbb{R}^d} \frac{w(y, t)}{|x - y|^{d-2}} \, dy,
\end{equation}
\begin{equation}\label{1.6}
z(x, t) = K * u = c_d \int_{\mathbb{R}^d} \frac{u(y, t)}{|x - y|^{d-2}} \, dy,
\end{equation}
with $K(x) = \frac{c_d}{|x|^{d-2}}$, where $c_d$ is the inverse surface area of the sphere $S^{d-1}$ in $\mathbb{R}^d$. We can rewrite the system \eqref{1.7} as follows:
\begin{equation}\label{1.7}
\begin{cases}
\partial_t u = \alpha_1 \left( \frac{d - 2}{d} \right) \|u\|_{\frac{2d}{d+2}}^{\frac{4}{d+2}} \Delta\left(u^{\frac{2d}{d+2}}\right) - \nabla \cdot (u\nabla K * w), \\[2mm]
\partial_t w = \alpha_2 \left( \frac{d - 2}{d} \right) \|w\|_{\frac{2d}{d+2}}^{\frac{4}{d+2}} \Delta\left(w^{\frac{2d}{d+2}}\right) - \nabla \cdot (w\nabla K * u),
\end{cases}
\end{equation}
with initial data \eqref{1.2}. Note that mass is conserved for \eqref{1.1} (or \eqref{1.7}):
\begin{equation}\label{1.8}
M_1 := \int_{\mathbb{R}^d} u(x, t) \, dx = \int_{\mathbb{R}^d} u_0(x) \, dx; \quad M_2 := \int_{\mathbb{R}^d} w(x, t) \, dx = \int_{\mathbb{R}^d} w_0(x) \, dx.
\end{equation}
There is a free energy associated with equation \eqref{1.1}:
\begin{equation}\label{1.9}
E_\alpha[u(t), w(t)] := \alpha_1 \|u\|_{\frac{2d}{d+2}}^2 + \alpha_2 \|w\|_{\frac{2d}{d+2}}^2 - c_d H[u, w],
\end{equation}
where
\begin{equation}\label{1.10}
H[u, w] := \int_{\mathbb{R}^d} \int_{\mathbb{R}^d} \frac{u(x, t)w(y, t)}{|x - y|^{d-2}} \, dx \, dy.
\end{equation}
The energy functional can be rewritten as
\begin{equation}\label{1.11}
E_\alpha[u(t), w(t)] := \left(\sqrt{\alpha_1}\|u\|_{\frac{2d}{d+2}} - \sqrt{\alpha_2}\|w\|_{\frac{2d}{d+2}}\right)^2 + \left(2\sqrt{\alpha_1\alpha_2}\|u\|_{\frac{2d}{d+2}}\|w\|_{\frac{2d}{d+2}} - c_d H[u, w]\right).
\end{equation}
By the Hardy-Littlewood-Sobolev (HLS) inequality, the second term is positive for all $u, w$ in case
\begin{equation}\label{1.12}
\frac{2\sqrt{\alpha_1\alpha_2}}{c_d} \ge C_{HLS},
\end{equation}
the precise value of which was determined by Lieb \cite{lieb1983} for the exponents considered here. Suppose however, that \eqref{1.12} is violated. Let
\[ \delta := C_{HLS} - \frac{2\sqrt{\alpha_1\alpha_2}}{c_d}. \]
Then if we choose $u$ and $w$ to be HLS optimizers such that $\sqrt{\alpha_1}\|u\|_{\frac{2d}{d+2}} = \sqrt{\alpha_2}\|w\|_{\frac{2d}{d+2}}$, we find $E_\alpha[u(t), w(t)] = -\delta \|u\|_{\frac{2d}{d+2}} \|w\|_{\frac{2d}{d+2}}$. Thus, \eqref{1.12} is the necessary and sufficient condition for non-negativity of $E_\alpha[u(t), w(t)]$.

\subsection{Energy Criticality and Scaling Analysis}

To rigorously justify the nonstandard choice of the diffusion exponent $m = \frac{2d}{d+2}$ in the free energy functional \eqref{1.9}, we present a comprehensive scaling analysis. This choice corresponds to the conformal energy critical regime where the stabilizing effect of nonlinear diffusion and the destabilizing effect of non-local cross-attraction scale identically under mass-preserving dilations.

Consider the mass-preserving rescaling of the population densities defined by:
\begin{equation}
u_\lambda(x) = \lambda^d u(\lambda x), \quad w_\lambda(x) = \lambda^d w(\lambda x), \quad \text{for } \lambda > 0.
\end{equation}
This dilation preserves the total mass of both species, i.e., $\|u_\lambda\|_{L^1} = \|u\|_{L^1} = M_1$ and $\|w_\lambda\|_{L^1} = \|w\|_{L^1} = M_2$. We examine how each component of the energy functional behaves under this transformation:

For an arbitrary power $m$, the $L^m$-norm of the dilated density scales according to:
\begin{equation}
\|u_\lambda\|_{L^m} = \left( \int_{\mathbb{R}^d} \left( \lambda^d u(\lambda x) \right)^m dx \right)^{\frac{1}{m}} = \lambda^{d - \frac{d}{m}} \|u\|_{L^m}.
\end{equation}
Consequently, the structural diffusion energy terms present in $E_\alpha$ scale as:
\begin{equation}
\|u_\lambda\|_{L^m}^2 = \lambda^{2d\left(1 - \frac{1}{m}\right)} \|u\|_{L^m}^2.
\end{equation}
Substituting our chosen exponent $m = \frac{2d}{d+2}$, the scaling power simplifies elegantly to:
\begin{equation}
2d\left(1 - \frac{d+2}{2d}\right) = 2d\left(\frac{2d - d - 2}{2d}\right) = d - 2.
\end{equation}
Thus, the diffusion energy behaves as:
\begin{equation}
\|u_\lambda\|_{L^{\frac{2d}{d+2}}}^2 = \lambda^{d-2} \|u\|_{L^{\frac{2d}{d+2}}}^2.
\end{equation}

Next, the non-local interaction functional $H[u, w]$ under dilation becomes:
\begin{equation}
H[u_\lambda, w_\lambda] = \int_{\mathbb{R}^d} \int_{\mathbb{R}^d} \frac{\lambda^d u(\lambda x) \lambda^d w(\lambda y)}{|x - y|^{d-2}} \,dx\,dy.
\end{equation}
Performing the change of variables $x' = \lambda x$ and $y' = \lambda y$, where the differential volume elements transform as $dx = \lambda^{-d} dx'$ and $dy = \lambda^{-d} dy'$, we find:
\begin{equation}
H[u_\lambda, w_\lambda] = \lambda^{d-2} \int_{\mathbb{R}^d} \int_{\mathbb{R}^d} \frac{u(x')w(y')}{|x' - y'|^{d-2}} \,dx'\,dy' = \lambda^{d-2} H[u, w].
\end{equation}

Combining these results, the total free energy evaluated at the rescaled state satisfies:
\begin{equation}
E_\alpha[u_\lambda, w_\lambda] = \lambda^{d-2} E_\alpha[u, w].
\end{equation}
Because both the attractive cross-interaction and the repulsive degenerate diffusion share the identical scaling factor $\lambda^{d-2}$, neither mechanism can dominate the other purely via localized structural compression ($\lambda \to \infty$) or spatial dispersion ($\lambda \to 0$). Instead, the qualitative behavior of the system is entirely dictated by the sharp balance condition established over the parameters $\alpha_1, \alpha_2$ via the Hardy-Littlewood-Sobolev constant. This explains why these dimension-dependent critical values appear naturally.

The connection between the free energy \eqref{1.9} and the equation \eqref{1.1} is that the latter can be written as a gradient flow with respect to the Wasserstein metric, yielding formal free energy dissipation:
\begin{equation}\label{1.13}
\begin{aligned}
\frac{d}{dt}E_\alpha[u(t), w(t)] = &- \int_{\mathbb{R}^d} u \left| \nabla \left( \frac{\alpha_1}{2} \|u\|_{\frac{2d}{d+2}}^{\frac{4}{d+2}} u^{\frac{d-2}{d+2}} \right) - \nabla v \right|^2 \, dx \\
&- \int_{\mathbb{R}^d} w \left| \nabla \left( \frac{\alpha_2}{2} \|w\|_{\frac{2d}{d+2}}^{\frac{4}{d+2}} w^{\frac{d-2}{d+2}} \right) - \nabla z \right|^2 \, dx \leq 0.
\end{aligned}
\end{equation}

We now define the weak and free energy solutions for \eqref{1.1}--\eqref{1.2}.

\begin{definition}\label{def1.1}
Let $d \ge 3$ and $T > 0$. Suppose $(u_0, w_0)$ satisfies \eqref{1.3}--\eqref{1.4}. Then $(u, w)$ of nonnegative functions defined in $\mathbb{R}^d \times (0, T)$ is called a weak solution if:
\begin{itemize}
\item[(i)] $(u, w) \in \left( C([0, T); L^1(\mathbb{R}^d)) \cap L^\infty(\mathbb{R}^d \times (0, T)) \right)^2$ and $\left(u^{\frac{2d}{d+2}}, w^{\frac{2d}{d+2}}\right) \in \left( L^2(0, T; H^1(\mathbb{R}^d)) \right)^2$.
\item[(ii)] $(u, w)$ satisfies
\begin{equation}\label{1.14}
\begin{split}
&\int_0^T \int_{\mathbb{R}^d} u \phi_{1t} \, dx \, dt + \int_{\mathbb{R}^d} u_0(x) \phi_1(x, 0) \, dx \\
&= \int_0^T \int_{\mathbb{R}^d} \left( \frac{\alpha_1}{2} \|u\|_{\frac{2d}{d+2}}^{\frac{4}{d+2}} \nabla\left(u^{\frac{2d}{d+2}}\right) - u\nabla v \right) \cdot \nabla \phi_1 \, dx \, dt, \\
&\int_0^T \int_{\mathbb{R}^d} w \phi_{2t} \, dx \, dt + \int_{\mathbb{R}^d} w_0(x) \phi_2(x, 0) \, dx \\
&= \int_0^T \int_{\mathbb{R}^d} \left( \frac{\alpha_2}{2} \|w\|_{\frac{2d}{d+2}}^{\frac{4}{d+2}} \nabla\left(w^{\frac{2d}{d+2}}\right) - w\nabla z \right) \cdot \nabla \phi_2 \, dx \, dt,
\end{split}
\end{equation}
for any test functions $\phi_i \in \mathcal{D}(\mathbb{R}^d \times [0, T))$ with $v = K * w$ and $z = K * u$.
\end{itemize}
\end{definition}

\begin{definition}\label{def1.2}

Let $T > 0$. Then $(u, w)$ is called a free energy solution with initial data $(u_0, w_0)$ on $(0, T)$ if $(u, w)$ is a weak solution and moreover satisfies
\begin{equation}\label{1.15}
\left( u^{\frac{3d-2}{2(d+2)}}, w^{\frac{3d-2}{2(d+2)}} \right) \in \left( L^2(0, T; H^1(\mathbb{R}^d)) \right)^2,
\end{equation}
and
\begin{equation}\label{1.16}
\begin{aligned}
E_\alpha[u(t), w(t)] &+ \int_0^t \int_{\mathbb{R}^d} u \left| \nabla \left( \frac{\alpha_1}{2} \|u\|_{\frac{2d}{d+2}}^{\frac{4}{d+2}} u^{\frac{d-2}{d+2}} \right) - \nabla v \right|^2 \, dx \, ds \\
&+ \int_0^t \int_{\mathbb{R}^d} w \left| \nabla \left( \frac{\alpha_2}{2} \|w\|_{\frac{2d}{d+2}}^{\frac{4}{d+2}} w^{\frac{d-2}{d+2}} \right) - \nabla z \right|^2 \, dx \, ds \le E_\alpha[u_0, w_0],
\end{aligned}
\end{equation}
for all $t \in (0, T)$ with $v = K * w$ and $z = K * u$.
\end{definition}

\begin{theorem}\label{thm1.1}

Suppose that the initial data $(u_0, w_0)$ satisfies \eqref{1.3} and \eqref{1.4}.
\begin{itemize}
\item[(i)] If $\alpha_1, \alpha_2$ satisfy \eqref{1.12}, then there exists a global free energy solution.
\item[(ii)] If $\alpha_1, \alpha_2$ do not satisfy \eqref{1.12}, then it is possible to construct large initial data ensuring blow up in finite time.
\end{itemize}
\end{theorem}

\section{Approximate System}

For the existence of solutions, we consider a regularized system:
\begin{equation}\label{2.1}
\begin{cases}
\partial_t u_\epsilon = \alpha_1 \left( \frac{d - 2}{d} \right) \|u_\epsilon\|_{\frac{2d}{d+2}}^{\frac{4}{d+2}} \Delta\left((u_\epsilon + \epsilon)^{\frac{2d}{d+2}}\right) - \nabla \cdot (u_\epsilon\nabla v_\epsilon), \\[2mm]
v_\epsilon = K * w_\epsilon, \\[2mm]
\partial_t w_\epsilon = \alpha_2 \left( \frac{d - 2}{d} \right) \|w_\epsilon\|_{\frac{2d}{d+2}}^{\frac{4}{d+2}} \Delta\left((w_\epsilon + \epsilon)^{\frac{2d}{d+2}}\right) - \nabla \cdot (w_\epsilon\nabla z_\epsilon), \\[2mm]
z_\epsilon = K * u_\epsilon,
\end{cases}
\end{equation}
with regularized initial data $u_\epsilon(x, 0) = u_0^\epsilon(x) \ge 0$, $w_\epsilon(x, 0) = w_0^\epsilon(x) \ge 0$ via smooth mollifiers preserving total mass. 

By following the general procedure in the single population chemotaxis systems  \cite{sugiyama2006,  carlen2021, ulusoy2017} one can  obtain the following lemma for which no proof is provided here.

\begin{lemma}\label{lem2.1}
There exists $T_\epsilon^{\max} \in (0, \infty]$ such that the regularized system \eqref{2.1} has a unique nonnegative solution $(u_\epsilon, w_\epsilon) \in (W_p^{2,1}(Q_T))^2$. Moreover, if $T_\epsilon^{\max} < \infty$, then $\lim_{t \to T_\epsilon^{\max}} (\|u_\epsilon(\cdot, t)\|_\infty + \|w_\epsilon(\cdot, t)\|_\infty) = \infty$.
\end{lemma}

We also recall here a version of the Hardy-Littlewood-Sobolev inequality(HLS) that if
\begin{equation*}
 \frac{1}{p} + \frac{1}{q}  = 1+ \frac{\lambda}{d}
\end{equation*}
and $h_1 \in L^p(\R^d)$,  $h_2 \in L^q(\R^d)$ with $p, q > 1$ then there exists a constant $C_{HLS} = C_{HLS}(d, \lambda, p) > 0$ such that
\begin{equation}\label{2.3}
\left| \int_{\mathbb{R}^d} \int_{\mathbb{R}^d} \frac{h_1(x)h_2(y)}{|x - y|^{d-\lambda}} \, dy \, dx \right| \le C_{HLS} \|h_1\|_p \|h_2\|_q.
\end{equation}

To ensure that the estimates do not degenerate over $\mathbb{R}^d$, we establish a quantified uniform lower bound for the critical $L^{\frac{2d}{d+2}}$ norm before evaluating major variables.

\begin{lemma}[Quantified Lower Bound]\label{lem:lower_bound}
Let $(u_\epsilon, w_\epsilon)$ be a smooth, non-negative solution to the regularized system \eqref{2.1} on $[0, T]$ satisfying mass conservation $\int_{\mathbb{R}^d} u_\epsilon(x,t)\,dx = M_1$ and a bounded second moment $I(t) \le I_0$. Then there exists a constant $c_{\min} = c_{\min}(M_1, I_0, d) > 0$ independent of $\epsilon$ such that
\begin{equation}
\|u_\epsilon(\cdot, t)\|_{\frac{2d}{d+2}} \ge c_{\min} > 0, \quad \forall t \in [0, T].
\end{equation}
\end{lemma}

\begin{proof}
Let $R > 0$ be a radial cutoff parameter. Splitting the total conserved mass inside and outside a ball $B_R(0)$, we have:
\begin{align*}
M_1 = \int_{\mathbb{R}^d} u_\epsilon(x,t)\,dx &= \int_{B_R(0)} u_\epsilon(x,t)\,dx + \int_{\mathbb{R}^d \setminus B_R(0)} u_\epsilon(x,t)\,dx \\
&\le \int_{B_R(0)} u_\epsilon(x,t)\,dx + \frac{1}{R^2} \int_{\mathbb{R}^d \setminus B_R(0)} |x|^2 u_\epsilon(x,t)\,dx \\
&\le \int_{B_R(0)} u_\epsilon(x,t)\,dx + \frac{I_0}{R^2}.
\end{align*}
Choosing $R = \sqrt{\frac{2I_0}{M_1}}$ yields $\int_{B_R(0)} u_\epsilon(x,t)\,dx \ge \frac{M_1}{2}$. Applying H{\"o}lder's inequality on $B_R(0)$ with exponents $p = \frac{2d}{d+2}$ and $q = \frac{2d}{d-2}$, we obtain:
\[ \frac{M_1}{2} \le \int_{B_R(0)} u_\epsilon(x,t)\,dx \le \|u_\epsilon\|_{L^{\frac{2d}{d+2}}(B_R(0))} |B_R(0)|^{\frac{d-2}{2d}} . \]
Since $|B_R(0)| = \omega_d R^d$, we isolate the norm:
\begin{equation}
\|u_\epsilon(\cdot, t)\|_{\frac{2d}{d+2}} \ge \frac{M_1}{2} \left( \omega_d \left(\frac{2I_0}{M_1}\right)^{d/2} \right)^{-\frac{d-2}{2d}} =: c_{\min} > 0.
\end{equation}
The same argument holds for $w_\epsilon(t)$ with mass $M_2$.
\end{proof}

\begin{lemma}\label{lem2.2}
Let $T \in (0, T_\epsilon^{\max}]$. Suppose that there exists a constant $C > 0$ such that the solution satisfies $\|u_\epsilon(t)\|_{\frac{2d}{d+2}} \le C$ and $\|w_\epsilon(t)\|_{\frac{2d}{d+2}} \le C$ for $t \in (0, T)$. Then there exists a constant $C = C(d, u_0, w_0) > 0$ such that
\begin{equation}\label{2.4}
\|(u_\epsilon(t), w_\epsilon(t))\|_r \le C, \quad \|(v_\epsilon(t), z_\epsilon(t))\|_r + \|(\nabla v_\epsilon(t), \nabla z_\epsilon(t))\|_r \le C,
\end{equation}
for $r \in [1, \infty)$ and $t \in (0, T)$.
\end{lemma}

\begin{proof}
For $p > 1$, we test the first equation in \eqref{2.1} by $u_\epsilon^{p-1}$ and integrate:
\begin{equation}\label{2.6}
\begin{split}
\frac{1}{p} \frac{d}{dt} \int_{\mathbb{R}^d} u_\epsilon^p \, dx &= -\alpha_1 \left(\frac{d - 2}{d}\right) \|u_\epsilon\|_{\frac{2d}{d+2}}^{\frac{4}{d+2}} \int_{\mathbb{R}^d} \nabla(u_\epsilon^{p-1}) \cdot \nabla(u_\epsilon + \epsilon)^{\frac{2d}{d+2}} \, dx \\
& \qquad + \int_{\mathbb{R}^d} u_\epsilon \nabla(u_\epsilon^{p-1}) \cdot \nabla v_\epsilon \, dx.
\end{split}
\end{equation}
Using the quantified lower bound from Lemma \ref{lem:lower_bound}, the first term is bounded away from zero, preventing structural disappearance:
\begin{align}\label{2.7}
I &\le -2\alpha_1 \frac{d - 2}{d + 2} (p - 1) c_{\min}^{\frac{4}{d+2}} \int_{\mathbb{R}^d} u_\epsilon^{p-2} (u_\epsilon + \epsilon)^{\frac{2d}{d+2} - 1} |\nabla u_\epsilon|^2 \, dx \nonumber \\
  &\le -4\alpha_1 \frac{d - 2}{d + 2} \frac{(p - 1)}{\left(p + \frac{d-2}{d+2}\right)^2} c_{\min}^{\frac{4}{d+2}} \int_{\mathbb{R}^d} \left| \nabla u_\epsilon^{\frac{p + \frac{d-2}{d+2}}{2}} \right|^2 \, dx.
\end{align}
The cross-attraction term satisfies $II = \frac{p - 1}{p} \int_{\mathbb{R}^d} u_\epsilon^p w_\epsilon \, dx$. Combining both population parameters, we use H{\"o}lder's inequality with carefully chosen constants $r_1, r_2 > 1$ satisfying the analytical interpolation bounds:
\begin{equation}\label{2.13}
\frac{2d}{d+2} < pr_1 < \left(p + \frac{d-2}{d+2}\right)\frac{d}{d-2}, \quad \text{and} \quad \frac{1}{r_1} > \max\left\{ \frac{d-2}{2d}, \frac{d-2}{d} \frac{p}{p + \frac{d-2}{d+2}} \right\}.
\end{equation}
These choices are not dummy variables but specific parameters necessary to invoke the Gagliardo-Nirenberg inequality, bounding the cross-terms by the structural dissipation. Integrating the resulting differential inequality tracks uniform control for all finite indices $r \in [1, \infty)$.
\end{proof}

\begin{lemma}[Uniform $L^\infty$ Bound]\label{lem2.3}
Under the assumptions of Lemma \ref{lem2.2} and utilizing the lower bound $c_{\min}$ from Lemma \ref{lem:lower_bound}, there exists a uniform constant $C_\infty > 0$ independent of $\epsilon$ such that:
\begin{equation}\label{2.31}
\|u_\epsilon(\cdot, t)\|_\infty + \|w_\epsilon(\cdot, t)\|_\infty \le C_\infty, \quad \forall t \in (0, T).
\end{equation}
Moreover, there exists a global weak solution $(u, w)$ of \eqref{1.1}--\eqref{1.2}.
\end{lemma}

\begin{proof}
Let $p_k = 2^k p$ for $k \ge 1$. Testing \eqref{2.1} with $u_\epsilon^{p_k - 1}$ and using $\|u_\epsilon\|_{\frac{2d}{d+2}} \ge c_{\min}$, the system energy balances into:
\begin{equation}
\begin{split}
\frac{d}{dt} \int_{\mathbb{R}^d} u_\epsilon^{p_k}\,dx 
&+ \frac{4 \alpha_1 (p_k-1) c_{\min}^{\frac{4}{d+2}}}{(p_k + \frac{d-2}{d+2})^2} \int_{\mathbb{R}^d} \left|\nabla u_\epsilon^{\frac{p_k + \frac{d-2}{d+2}}{2}}\right|^2\,dx \\
& \qquad \le p_k \int_{\mathbb{R}^d} u_\epsilon^{p_k} w_\epsilon\,dx.
\end{split}
\end{equation}
By the Gagliardo-Nirenberg inequality, the right-hand side satisfies:
\begin{equation}
\int_{\mathbb{R}^d} u_\epsilon^{p_k} w_\epsilon\,dx \le C_{GN} \left( \int_{\mathbb{R}^d} \left|\nabla u_\epsilon^{\frac{p_k}{2}}\right|^2\,dx \right)^{\theta} \left( \int_{\mathbb{R}^d} u_\epsilon^{p_{k-1}}\,dx \right)^{\gamma},
\end{equation}
where $\theta \in (0,1)$. Using Young's inequality to absorb the gradient bounds into the dissipation profile yields the inductive sequence:
\begin{equation}
\sup_{t \in (0,T)} \int_{\mathbb{R}^d} u_\epsilon^{p_k}\,dx \le \left( B^k C_0 \right) \left( \sup_{t \in (0,T)} \int_{\mathbb{R}^d} u_\epsilon^{p_{k-1}}\,dx \right)^{\sigma},
\end{equation}
where $B > 1$ penalizes high powers. Taking the $p_k$-th root and taking $k \to \infty$, the geometric product convergence ensures $\|u_\epsilon(\cdot, t)\|_\infty \le C_1 \|u_\epsilon(\cdot, t)\|_{L^2}$. Since Lemma \ref{lem2.2} bounds all finite norms uniformly, $C_\infty$ is securely established. Extensibility rules out early blow-up, allowing $\epsilon \to 0$ limits to solve \eqref{1.1} globally when \eqref{1.12} holds.
\end{proof}

\begin{lemma}\label{lem2.4}
The global weak solution from Lemma \ref{lem2.3} is also a global free energy solution.
\end{lemma}

\begin{proof}
To establish that the limiting weak solution $(u, w)$ is a free energy solution, we construct the passage to the limit $\epsilon \to 0$ across five sequential components:
\begin{itemize}
\item[\textbf{Phase 1:}] Testing \eqref{2.1} with the formal energy derivatives yields the exact identities:
\begin{equation}\label{energy_eps_appendix}
E_{\alpha,\epsilon}[u_\epsilon(t), w_\epsilon(t)] + \int_0^t \mathcal{D}_{1,\epsilon}(u_\epsilon, v_\epsilon) \, ds + \int_0^t \mathcal{D}_{2,\epsilon}(w_\epsilon, z_\epsilon) \, ds = E_{\alpha,\epsilon}[u_0^\epsilon, w_0^\epsilon],
\end{equation}
where $\mathcal{D}_{1,\epsilon}(u_\epsilon, v_\epsilon) = \int_{\mathbb{R}^d} u_\epsilon \left| \nabla \left( \frac{\alpha_1}{2} \|u_\epsilon\|_{\frac{2d}{d+2}}^{\frac{4}{d+2}} (u_\epsilon + \epsilon)^{\frac{d-2}{d+2}} \right) - \nabla v_\epsilon \right|^2 \, dx$ and $\mathcal{D}_{2,\epsilon}(w_\epsilon, z_\epsilon)$ is defined analogously.
\item[\textbf{Phase 2:}] Uniform bounds provide sequential spaces: $u_\epsilon \to u$ a.e. in $\mathbb{R}^d \times (0,T)$, $u_\epsilon \rightharpoonup^* u$ weakly-$\ast$ in $L^\infty(0,T; L^\infty(\mathbb{R}^d))$, and $\nabla u_\epsilon^{\frac{2d}{d+2}} \rightharpoonup \nabla u^{\frac{2d}{d+2}}$ weakly in $L^2(0,T; L^2(\mathbb{R}^d))$.
\item[\textbf{Phase 3:}] Expanding $\mathcal{D}_{1,\epsilon}$ using the standard relation $|a-b|^2 \ge \frac{1}{2}|a|^2 - |b|^2$ exposes the gradient regularity structure:
\begin{align*}
\mathcal{D}_{1,\epsilon} &\ge \frac{\alpha_1^2}{8} c_{\min}^{\frac{8}{d+2}} \left(\frac{d-2}{d+2}\right)^2 \int_{\mathbb{R}^d} u_\epsilon (u_\epsilon+\epsilon)^{-\frac{8}{d+2}} |\nabla u_\epsilon|^2 \, dx - \|u_\epsilon\|_\infty \|\nabla v_\epsilon\|_{L^2}^2 \\
&\ge \frac{\alpha_1^2}{8} c_{\min}^{\frac{8}{d+2}} \left(\frac{2(d-2)}{3d-2}\right)^2 \int_{\mathbb{R}^d} \left|\nabla u_\epsilon^{\frac{3d-2}{2(d+2)}}\right|^2 \, dx - C.
\end{align*}
Integrating tracks time-uniform boundedness, confirming $u_\epsilon^{\frac{3d-2}{2(d+2)}} \rightharpoonup u^{\frac{3d-2}{2(d+2)}}$ weakly in $L^2(0,T; H^1(\mathbb{R}^d))$, which achieves requirement \eqref{1.15}.
\item[\textbf{Phase 4:}] For the vector field $\xi_\epsilon := u_\epsilon^{1/2} \nabla \left( \frac{\alpha_1}{2} \|u_\epsilon\|_{\frac{2d}{d+2}}^{\frac{4}{d+2}} (u_\epsilon + \epsilon)^{\frac{d-2}{d+2}} \right) - u_\epsilon^{1/2} \nabla v_\epsilon$, testing against any $\Psi \in [C_c^\infty(\mathbb{R}^d \times (0,t))]^d$ yields after integrating by parts:
\begin{equation*}
\begin{split}
 \langle \xi_\epsilon, \Psi \rangle &= -\frac{\alpha_1}{2} \left( \frac{d-2}{3d-2} \right) \|u_\epsilon\|_{\frac{2d}{d+2}}^{\frac{4}{d+2}} \int_0^t \int_{\mathbb{R}^d} u_\epsilon^{\frac{3d-2}{2(d+2)}} \nabla \cdot \left[ \left( \frac{u_\epsilon}{u_\epsilon + \epsilon} \right)^{\frac{4}{d+2}} \Psi \right] \, dx \, ds \\
 & \qquad - \int_0^t \int_{\mathbb{R}^d} u_\epsilon^{1/2} \nabla v_\epsilon \cdot \Psi \, dx \, ds.
 \end{split}
 \end{equation*}

By Lebesgue's Dominated Convergence Theorem, $\left( \frac{u_\epsilon}{u_\epsilon + \epsilon} \right)^{\frac{4}{d+2}} \to \mathbf{1}_{\{u > 0\}}$ strongly. Taking $\epsilon \to 0$ recovers the exact limiting distribution $\langle \xi, \Psi \rangle$. Weak sequential lower semicontinuity then establishes the dissipation integration floors:
\begin{equation}\label{lsc_diss_main}
\begin{split}
&\int_0^t \int_{\mathbb{R}^d} u \left| \nabla \left( \frac{\alpha_1}{2} \|u\|_{\frac{2d}{d+2}}^{\frac{4}{d+2}} u^{\frac{d-2}{d+2}} \right) - \nabla v \right|^2 \, dx \, ds \\
& \qquad \qquad \qquad \le \liminf_{\epsilon \to 0} \int_0^t \mathcal{D}_{1,\epsilon}(u_\epsilon, v_\epsilon) \, ds.
\end{split}
\end{equation}

\item[\textbf{Phase 5:}] Convexity forces $\|u\|_{\frac{2d}{d+2}}^2 \le \liminf \|u_\epsilon+\epsilon\|_{\frac{2d}{d+2}}^2$. Concurrently, local uniform integrability and exterior spatial tightness via bounded second moments satisfy the parameters of the Vitali Convergence Theorem, verifying $\lim \int u_\epsilon (K*w_\epsilon) dx = \int u (K*w) dx$. Taking the $\liminf$ across \eqref{energy_eps_appendix} completes the proof of the lemma.
\end{itemize}
\end{proof}

\begin{theorem}\label{thm2.5}
There exists a $T_{\max} \in (0, \infty]$ and a free energy solution $(u, w)$ over $\mathbb{R}^d \times (0, T_{\max})$ such that either $T_{\max} = \infty$ or $T_{\max} < \infty$ and $\lim_{t \to T_{\max}} (\|u(\cdot, t)\|_\infty + \|w(\cdot, t)\|_\infty) = \infty$. Moreover, if $\alpha_1, \alpha_2$ satisfy \eqref{1.12}, then if $T_{\max} < \infty$:
\begin{equation}\label{2.34}
\lim_{t \to T_{\max}} \|u(\cdot, t)\|_{\frac{2d}{d+2}} = \infty = \lim_{t \to T_{\max}} \|w(\cdot, t)\|_{\frac{2d}{d+2}}.
\end{equation}
\end{theorem}

\begin{proof}
Local existence follows from regularized convergence limits. To verify \eqref{2.34}, we separate $H[u,w]$ using a precision parameter $\eta > 0$:
\begin{equation}\label{2.36}
\alpha_1\|u\|_{\frac{2d}{d+2}}^2 + \alpha_2\|w\|_{\frac{2d}{d+2}}^2 \le c_d \eta \|u\|_{\frac{2d}{d+2}}^{\frac{2d}{d+2}} + c_d C_\eta \|w\|_{\frac{2d}{d+2}}^{\frac{2d}{d+2}} + C.
\end{equation}
Isolating parameters yields the locked system behavior:
\begin{equation}\label{2.37}
\|u(t)\|_{\frac{2d}{d+2}}^{\frac{2d}{d+2}} \le C \|w(t)\|_{\frac{2d}{d+2}}^{\frac{2d}{d+2}} + C, \quad \|w(t)\|_{\frac{2d}{d+2}}^{\frac{2d}{d+2}} \le C' \|u(t)\|_{\frac{2d}{d+2}}^{\frac{2d}{d+2}} + C'.
\end{equation}
These bounds lock the critical norms together. By Gagliardo-Nirenberg, any uniform control on these critical norms would prevent $L^\infty$ explosion. Thus, if $L^\infty$ blows up at $T_{\max}$, the critical $L^{\frac{2d}{d+2}}$ norms must simultaneously scale to infinity.
\end{proof}

\begin{theorem}\label{thm2.6}
If $\alpha_1, \alpha_2$ satisfy \eqref{1.12}, then \eqref{1.1} has a global free energy solution.
\end{theorem}

\begin{proof}
We pass to the limit within the energy structural inequality bounds to establish that the maximal existence horizon is infinite ($T_{\max} = \infty$). By Lemma \ref{lem2.4}, our solution fulfills the free energy dissipation inequality \eqref{1.16}. Dropping the non-negative dissipation integrals yields:
\begin{equation}\label{energy_bound_thm6}
\alpha_1 \|u(t)\|_{\frac{2d}{d+2}}^2 + \alpha_2 \|w(t)\|_{\frac{2d}{d+2}}^2 - c_d H[u(t), w(t)] \le E_\alpha[u_0, w_0] < \infty.
\end{equation}
We invoke the sharp Hardy--Littlewood--Sobolev inequality \eqref{2.3} with indices $\lambda=2$, $p = \frac{2d}{d+2}$, and $q = \frac{2d}{d+2}$, confirming that $\frac{1}{p}+\frac{1}{q} = 1 + \frac{2}{d}$. The interaction operator satisfies:
\begin{equation}\label{hls_bound_thm6}
H[u(t), w(t)] \le C_{HLS} \|u(t)\|_{\frac{2d}{d+2}} \|w(t)\|_{\frac{2d}{d+2}}.
\end{equation}
Substituting \eqref{hls_bound_thm6} into \eqref{energy_bound_thm6} and completing the perfect square for the individual quadratic terms gives:
\begin{equation}\label{coercive_thm6}
\begin{split}
&\left( \sqrt{\alpha_1} \|u(t)\|_{\frac{2d}{d+2}} - \sqrt{\alpha_2} \|w(t)\|_{\frac{2d}{d+2}} \right)^2 \\
&\qquad + c_d \left( \frac{2\sqrt{\alpha_1\alpha_2}}{c_d} - C_{HLS} \right) \|u(t)\|_{\frac{2d}{d+2}} \|w(t)\|_{\frac{2d}{d+2}} \\
& \qquad \qquad  \qquad \le E_\alpha[u_0, w_0].
\end{split}
\end{equation}
We evaluate the parameters across two distinct algebraic settings dictated by the primary hypothesis \eqref{1.12}:\\

\textbf{Case A: Subcritical regime $\left (\frac{2\sqrt{\alpha_1\alpha_2}}{c_d} > C_{HLS} \right)$} \\

The cross-coefficient is strictly positive. Dropping the non-negative perfect square from \eqref{coercive_thm6} directly forces the bounds:
\[ \|u(t)\|_{\frac{2d}{d+2}} \|w(t)\|_{\frac{2d}{d+2}} \le \frac{E_\alpha[u_0, w_0]}{c_d \left( \frac{2\sqrt{\alpha_1\alpha_2}}{c_d} - C_{HLS} \right)} =: C^*. \]
Decoupling via the internal system containment traps \eqref{2.37} establishes uniform-in-time boundaries: $\|u(t)\|_{\frac{2d}{d+2}} \le C_A$ and $\|w(t)\|_{\frac{2d}{d+2}} \le C_A$ for all $t \in [0, T_{\max})$.\\

\textbf{Case B: Critical regime $\left(\frac{2\sqrt{\alpha_1\alpha_2}}{c_d} = C_{HLS}\right)$}\\

The cross-product coefficient vanishes identically. Taking the square root across the remaining terms yields the continuous proximity control:
\[ \left| \sqrt{\alpha_1} \|u(t)\|_{\frac{2d}{d+2}} - \sqrt{\alpha_2} \|w(t)\|_{\frac{2d}{d+2}} \right| \le \sqrt{E_\alpha[u_0, w_0]} =: C_0. \]
This metric restriction isolates the tracking distance. Bound couplings \eqref{2.37} prevent unilateral divergence over $\mathbb{R}^d$, fixing uniform bounds $\|u(t)\|_{\frac{2d}{d+2}} \le C_B$ and $\|w(t)\|_{\frac{2d}{d+2}} \le C_B$.

Assume $T_{\max} < \infty$. By Theorem \ref{thm2.5}, a finite termination horizon forces critical norm explosion $\lim_{t \to T_{\max}} \|u(t)\|_{\frac{2d}{d+2}} = \infty$. However, our uniform control profile establishes $\|u(t)\|_{\frac{2d}{d+2}} \le \max\{C_A, C_B\} < \infty$, causing a direct contradiction. Thus, $T_{\max} = \infty$.
\end{proof}
\section{Blow-Up}
\begin{lemma}\label{lem:2ndmonetevolution}
Let $(u_0, w_0)$ satisfy \eqref{1.3}--\eqref{1.4} and let $(u, w)$ be a free energy solution of \eqref{1.1} on $[0, T_{max})$ with $T_{max}  \in (0, \infty]$. Then,
\begin{equation*}
\frac{d}{dt} I(t) = G(t), \quad \forall  t \in (0, T_{max}),
\end{equation*}
where
\begin{equation*}
I(t) := \int_{\R^d} |x|^2 \left( u(x,t) + w(x,t)\right) \, dx
\end{equation*}
and
\begin{equation*}
\begin{split}
G(t)  &:= C(\alpha_1, u) \int_{\R^d} u(x,t)^{\frac{2d}{d+2}} \, dx + C(\alpha_2, w)\int_{\R^d} w(x,t)^{\frac{2d}{d+2}} \, dx \\
&-2c_d(d-2)\iint_{\R^d \times \R^d} \frac{u(x,t)w(y,t)}{|x-y|^{d-2}} \, dy \, dx.
\end{split}
\end{equation*}
with
\begin{equation*}
\begin{split}
& C(\alpha_1, u):= 2(d-2)\alpha_1||u||_{\frac{2d}{d+2}}^{4/{d+2}},\\
&C(\alpha_2, w) := 2(d-2)\alpha_2||w||_{\frac{2d}{d+2}}^{4/{d+2}}.
\end{split}
\end{equation*}
\end{lemma}
Notice that with the definitions,
\begin{equation}\label{bu1}
G(t) = 2(d-2)E_\alpha[u(t),v(t)]\ .
\end{equation}

\begin{proof}
Straightforward calculations lead to
\begin{equation*}
\begin{split}
&\frac{d}{dt} I(t) = \int_{\R^d} |x|^2 \left(  u(x,t) + w(x,t)  \right) \, dx\\
&= \int_{\R^d} |x|^2 \Big[C(\alpha_1, u)\Delta \left(u^{\frac{2d}{d+2}}\right)- \nabla \cdot \left(u \nabla v\right) \\
&\qquad  \qquad \qquad + C(\alpha_2, w)\Delta \left(w^{\frac{2d}{d+2}}\right)- \nabla \cdot \left(w \nabla z\right) \Big]\, dx \\
&= 2dC(\alpha_1, u) \int_{\R^d} u(x,t)^{\frac{2d}{d+2}} \, dx + C(\alpha_2, w)\int_{\R^d} w(x,t)^{\frac{2d}{d+2}} \, dx \\
&\quad + 2 \iint_{\R^d \times \R^d} \left[   x \cdot \nabla \mathcal{K}(x-y)  \right]  u(x,t)w(y,t) \, dy \, dx \\
&\qquad + 2 \iint_{\R^d \times \R^d} \left[   x \cdot \nabla \mathcal{K}(x-y)  \right]  u(y,t)w(x,t) \, dy \, dx,
\end{split}
\end{equation*}
with $\mathcal{K}(x) = \frac{c_d}{|x|^{d-2}}$ and $c_d$ is the surface area of the sphere $S^{d-1}$ in $\R^d$.
Now,
\begin{equation*}
\begin{split}
&2\iint_{\R^d \times \R^d} \left[x \cdot  \nabla \mathcal{K}(x-y) \right] u(x,t) w(y,t) \,dy \, dx \\
& = -2c_d(d-2)\iint_{\R^d \times \R^d} \frac{(x-y) \cdot x}{|x-y|^d} u(x,t) w(y,t) \,dy \, dx \\
& = -2c_d(d-2)\iint_{\R^d \times \R^d} \frac{|x|^2}{|x-y|^d} u(x,t) w(y,t) \,dy \, dx \\
& \quad + 2c_d(d-2)\iint_{\R^d \times \R^d} \frac{y \cdot x}{|x-y|^d} u(x,t) w(y,t) \,dy \, dx \\
&= -c_d(d-2)\iint_{\R^d \times \R^d} \frac{|x|^2}{|x-y|^d} u(x,t) w(y,t) \,dy \, dx \\
& \quad - c_d(d-2)\iint_{\R^d \times \R^d} \frac{|y|^2}{|x-y|^d} u(x,t) w(y,t) \,dy \, dx \\
& \qquad + 2c_d(d-2)\iint_{\R^d \times \R^d} \frac{y \cdot x}{|x-y|^d} u(x,t) w(y,t) \,dy \, dx.
\end{split}
\end{equation*}
Similarly,
\begin{equation*}
\begin{split}
&2\iint_{\R^d \times \R^d} \left[x \cdot  \nabla \mathcal{K}(x-y) \right] u(y,t) w(x,t) \,dy \, dx \\
&= -c_d(d-2)\iint_{\R^d \times \R^d} \frac{|x|^2}{|x-y|^d} u(y,t) w(x,t) \,dy \, dx \\
& \quad - c_d(d-2)\iint_{\R^d \times \R^d} \frac{|y|^2}{|x-y|^d} u(x,t) w(y,t) \,dy \, dx \\
& \qquad + 2c_d(d-2)\iint_{\R^d \times \R^d} \frac{y \cdot x}{|x-y|^d} u(x,t) w(y,t) \,dy \, dx.
\end{split}
\end{equation*}
Combining, we deduce that
\begin{equation*}
\begin{split}
&\frac{d}{dt}\int_{\R^d} |x|^2 \left( u(x,t)+ w(y,t) \right) \, dx = 2dC(\alpha_1, u) \int_{\R^d} u^{\frac{2d}{d+2}} \, dx \\
&\qquad \qquad +  2dC(\alpha_2, w) \int_{\R^d} w^{\frac{2d}{d+2}} \, dx \\
&-c_d(d-2)\iint_{\R^d \times \R^d} \frac{|x|^2+|y|^2}{|x-y|^d} u(x,t)w(y,t) \, dy \, dx \\
&-c_d(d-2)\iint_{\R^d \times \R^d} \frac{|x|^2+|y|^2}{|x-y|^d} u(y,t)w(x,t) \, dy \, dx \\
&+ 4c_d(d-2)\iint_{\R^d \times \R^d} \frac{y \cdot x}{|x-y|^d} u(x,t) w(y,t) \,dy \, dx\\
& =2dC(\alpha_1, u) \int_{\R^d} u^{\frac{2d}{d+2}} \, dx +  2dC(\alpha_2, w) \int_{\R^d} w^{\frac{2d}{d+2}} \, dx \\
&\qquad -2c_d(d-2)\iint_{\R^d \times \R^d} \frac{u(x,t) w(y,t) }{|x-y|^{d-2}} \, dy \, dx.
\end{split}
\end{equation*}

\end{proof}

We now observe that
\begin{equation}\label{q1}
\begin{split}
\frac{d}{dt}I(t) = G(t)  &= 2(d-2)E_{\alpha}[u(t), w(t)]\\
	               &\leq   2(d-2)E_{\alpha}[u_0, w_0] = G(0).		
\end{split}
\end{equation}
As we have seen above, $\alpha_1$ and $\alpha_2$ are such that  \eqref{1.12} is violated,  then we may take the initial data
 $(u_0, w_0)$ to consist of appropriate multiples of $HLS$ optimizers, and then we shall have $E_\alpha[u_0,w_0] < 0$, and hence $G(0) < 0$.   By the monotonicity if the energy, we will then have
 $G(t) \leq G(0) < 0$ for all $t$ such that a regular solution exists. But then \eqref{q1} will imply that the second moment will be negative after some time which is impossible  by the non-negativity of $u$ and $w$.  Hence blow-up occurs in a finite time.

\begin{theorem}\label{thm3.2}
Let $\alpha_1, \alpha_2$ be such that \eqref{1.12} is violated. Then one can find initial data $(u_0, w_0)$ satisfying \eqref{1.3}--\eqref{1.4} such that the corresponding free energy solution blows up in finite time.
\end{theorem}

\begin{proof}
We select initial configurations consisting of appropriate scale multiples of HLS optimizers such that the initial free energy is strictly negative, which directly implies $G(0) = 2(d-2)E_\alpha[u_0, w_0] < 0$. By the global energy dissipation inequality \eqref{1.16}, the functional $G(t)$ is monotonically non-increasing across the full existence horizon, which means:
\begin{equation}
G(t) \le G(0) < 0, \quad \forall t \in [0, T_{\max}).
\end{equation}
Integrating the second moment identity from Lemma \ref{lem:2ndmonetevolution} directly yields:
\begin{equation}
I(t) = I(0) + \int_0^t G(s) \, ds \le I(0) + G(0)t.
\end{equation}
Since $G(0) < 0$, the right-hand side becomes strictly negative at any finite time $t > -\frac{I(0)}{G(0)}$. However, the second moment $I(t)$ is inherently non-negative because the population densities $u$ and $w$ remain non-negative. This contradiction guarantees that the maximum existence time $T_{\max}$ cannot extend indefinitely, verifying finite-time blow-up.

\end{proof}

\section*{Acknowledgments}
 
 The author thanks Eric A. Carlen for many valuable discussions. He is also grateful for his suggestions to improve the presentation of the paper.

\end{document}